\documentclass[10pt]{amsart}

\usepackage[T1]{fontenc}
\usepackage{amsmath,amssymb,mathtools}
\usepackage{array,booktabs}
\usepackage{capt-of}
\usepackage{hyperref}

\hypersetup{hidelinks}
\allowdisplaybreaks
\numberwithin{equation}{section}

\newtheorem{theorem}{Theorem}[section]
\newtheorem{proposition}[theorem]{Proposition}
\newtheorem{lemma}[theorem]{Lemma}
\newtheorem{corollary}[theorem]{Corollary}

\title[Rank-two commutators in $M_6(\mathbb C)$]
{The $6\times6$ equality case of matrix spaces with rank-two commutators}
\author[Zhi-Lin Zhang]{Zhi-Lin Zhang}
\address{Independent Researcher, Taipei, Taiwan}

\email{hsa00000@gmail.com}
\date{}

\begin{document}

\begin{abstract}
Let $\mathcal V\subseteq M_6(\mathbb C)$ be a $17$-dimensional linear subspace such that
\[
\operatorname{rank}[S,T]\leq2
\quad(S,T\in\mathcal V).
\]
We prove that $\mathcal V$, or its transpose, is conjugate to the algebra
\[
\left\{
\begin{pmatrix}
A&B&C\\
0&\lambda I_2&D\\
0&0&\lambda I_2
\end{pmatrix}:
A,B,C,D\in M_2(\mathbb C),\ \lambda\in\mathbb C
\right\}.
\]
Consequently, the corresponding closed algebraic locus in
$\operatorname{Gr}(17,M_6(\mathbb C))$ is the disjoint union of two
nonsingular irreducible components, each isomorphic to
$\operatorname{Fl}(2,4;6)$.  We also prove that the Zariski tangent space
at $\mathcal A$ of the corresponding closed algebraic locus is equal to the
tangent space to the conjugacy orbit of $\mathcal A$.
\end{abstract}

\maketitle

\section{Introduction}

The case $k=0$ is classical.  Schur determined the maximal dimension of a
commutative subspace of $M_n(\mathbb C)$ and the extremal spaces
\cite{Schur}; see also Jacobson's later treatment \cite{Jacobson}.
Omladi\v{c}, Radjavi, and \v{S}ivic proved \cite[Theorem~12]{ORS} that a
linear subspace $\mathcal V\subseteq M_n(\mathbb C)$ satisfying
\[
\operatorname{rank}[S,T]\leq k
\quad(S,T\in\mathcal V)
\]
has dimension at most
\[
nk+\left\lfloor\frac{(n-k)^2}{4}\right\rfloor+1.
\]
They conjectured a classification of the equality cases
\cite[Conjecture~5]{ORS}.  They proved this equality-case classification for
$k=1$ and $k=n-1$, and for arbitrary $k$ under the additional assumption that
$\mathcal V$ is an algebra \cite[Theorems~2, 4, and~13]{ORS}.  The main result
of this paper proves the equality-case classification in their conjecture for
\[
(n,k)=(6,2),
\]
for which the dimension bound is $17$.

Define
\begin{equation}\label{eq:A-definition}
\mathcal A=
\left\{
S(A,B,C,D,\lambda)
=
\begin{pmatrix}
A&B&C\\
0&\lambda I_2&D\\
0&0&\lambda I_2
\end{pmatrix}:
A,B,C,D\in M_2(\mathbb C),\ \lambda\in\mathbb C
\right\}.
\end{equation}
Then $\dim_{\mathbb C}\mathcal A=17$.  For a subspace
$\mathcal V\subseteq M_6(\mathbb C)$, write
\[
\mathcal V^{\mathsf T}
=
\{S^{\mathsf T}:S\in\mathcal V\}.
\]

\begin{theorem}\label{thm:main}
Let $\mathcal V\subseteq M_6(\mathbb C)$ be a $17$-dimensional complex linear
subspace satisfying
\[
\operatorname{rank}[S,T]\leq2
\quad(S,T\in\mathcal V).
\]
There exists $P\in\operatorname{GL}_6(\mathbb C)$ such that either
\[
P^{-1}\mathcal VP=\mathcal A
\quad\text{or}\quad
P^{-1}\mathcal V^{\mathsf T}P=\mathcal A.
\]
\end{theorem}

The proof has three parts.  First, the conjugacy orbit of $\mathcal A$ is
identified with $\operatorname{Fl}(2,4;6)$.  Second, we compute the Zariski
tangent space at $\mathcal A$ of the closed algebraic set of such subspaces and
show that it is exactly the tangent space to this orbit.  Finally, following the
projective Borel-fixed-point reduction of \cite[Lemmas~7 and~8]{ORS}, we reduce
every irreducible component to the upper-triangular fixed-point classification
in \cite[Proposition~11]{ORS}.  After the global classification is obtained, we
distinguish $\mathcal O$ from $\mathcal O^{\mathsf T}$ by an intrinsic invariant
of the corresponding matrix algebras.

\section{The algebraic locus and the conjugacy orbit}

We first record that $\mathfrak X$ is a projective algebraic subset of the
Grassmannian.  We then determine the conjugacy orbit of $\mathcal A$ and
identify it with a flag variety.

Let
\begin{equation}\label{eq:locus}
\mathfrak X=
\left\{
\mathcal V\in\operatorname{Gr}(17,M_6(\mathbb C)):
\operatorname{rank}[S,T]\leq2
\text{ for all }S,T\in\mathcal V
\right\}.
\end{equation}
Throughout, algebraic sets and Zariski tangent spaces are understood in the
classical sense over $\mathbb C$, using affine charts and vanishing ideals.

\begin{proposition}\label{prop:closed}
The subset $\mathfrak X$ is a closed algebraic subset of
$\operatorname{Gr}(17,M_6(\mathbb C))$ and hence is a projective algebraic
set.
\end{proposition}

\begin{proof}
This is \cite[Lemma~7]{ORS} with $(n,k,m)=(6,2,17)$.
\end{proof}

We next determine the geometry of the orbit of $\mathcal A$.
Let $\operatorname{GL}_6(\mathbb C)$ act by conjugation.  With respect to the
standard basis $e_1,\ldots,e_6$, set
\[
E_1=\operatorname{span}(e_1,e_2),
\quad
E_2=\operatorname{span}(e_3,e_4),
\quad
E_3=\operatorname{span}(e_5,e_6).
\]
Thus $\mathbb C^6=E_1\oplus E_2\oplus E_3$.  Define
\[
\mathcal O=\operatorname{GL}_6(\mathbb C)\cdot\mathcal A,
\quad
\mathcal O^{\mathsf T}=\operatorname{GL}_6(\mathbb C)\cdot\mathcal A^{\mathsf T}.
\]

\begin{proposition}\label{prop:orbit}
The algebra $\mathcal A$ belongs to $\mathfrak X$, and its conjugation normalizer
is
\[
N_{\operatorname{GL}_6(\mathbb C)}(\mathcal A)
=
\operatorname{Stab}_{\operatorname{GL}_6(\mathbb C)}(E_1\subsetneq E_1\oplus E_2).
\]
Consequently,
\[
\mathcal O\cong\operatorname{Fl}(2,4;6),
\quad
\dim\mathcal O=12.
\]
In particular, $\mathcal O$ is nonsingular, projective, irreducible, and closed
in the Grassmannian.
\end{proposition}

\begin{proof}
For $S,T\in\mathcal A$, direct block multiplication gives
\[
[S,T]
=
\begin{pmatrix}
*&*&*\\
0&0&0\\
0&0&0
\end{pmatrix}.
\]
Hence $\operatorname{rank}[S,T]\leq2$, so $\mathcal A\in\mathfrak X$.

Let $J=\operatorname{rad}(\mathcal A)$.  Explicitly,
\[
J=
\left\{
\begin{pmatrix}
0&B&C\\
0&0&D\\
0&0&0
\end{pmatrix}:
B,C,D\in M_2(\mathbb C)
\right\},
\]
and
\[
J^2=
\left\{
\begin{pmatrix}
0&0&C\\
0&0&0\\
0&0&0
\end{pmatrix}:
C\in M_2(\mathbb C)
\right\}.
\]
Indeed, $J$ is nilpotent and
$\mathcal A/J\cong M_2(\mathbb C)\oplus\mathbb C$ is semisimple.  Moreover,
\begin{equation}\label{eq:intrinsic-flag}
\sum_{R\in J^2}\operatorname{im}R=E_1,
\quad
\bigcap_{R\in J^2}\ker R=E_1\oplus E_2.
\end{equation}
Every element of $N_{\operatorname{GL}_6(\mathbb C)}(\mathcal A)$ preserves $J^2$ under conjugation and
therefore preserves the two subspaces in \eqref{eq:intrinsic-flag}.  Thus
\[
N_{\operatorname{GL}_6(\mathbb C)}(\mathcal A)
\subseteq
\operatorname{Stab}_{\operatorname{GL}_6(\mathbb C)}(E_1\subsetneq E_1\oplus E_2).
\]

Conversely, an element of the stabilizer on the right is block upper triangular
with respect to $E_1\oplus E_2\oplus E_3$.  Conjugating an element of $\mathcal A$
by such a matrix preserves block upper triangularity, and the last two diagonal
blocks remain the same scalar matrix.  Hence the stabilizer normalizes
$\mathcal A$, proving equality.

The orbit map therefore induces
\[
\mathcal O
\cong
\operatorname{GL}_6(\mathbb C)/N_{\operatorname{GL}_6(\mathbb C)}(\mathcal A)
\cong
\operatorname{Fl}(2,4;6).
\]
Finally,
\[
\dim\operatorname{Fl}(2,4;6)
=
\dim\operatorname{Gr}(2,6)+\dim\operatorname{Gr}(2,4)
=8+4=12.
\]
The remaining assertions are standard properties of flag varieties.
\end{proof}

\section{Tangent-space computation}

The purpose of this section is to prove
\[
T_{\mathcal A}\mathfrak X=T_{\mathcal A}\mathcal O.
\]
We first introduce affine coordinates on the Grassmannian and compute the
tangent space to $\mathcal O$ at $\mathcal A$.  We then linearize the rank-two
condition and show that every tangent vector to $\mathfrak X$ at $\mathcal A$
lies in $T_{\mathcal A}\mathcal O$.

\subsection{Tangent coordinates}

We use the complement
\begin{equation}\label{eq:complement}
\mathcal W=
\left\{
\begin{pmatrix}
0&0&0\\
P&U&0\\
Q&R&V
\end{pmatrix}:
P,Q,R,U,V\in M_2(\mathbb C),\
\operatorname{tr}U+\operatorname{tr}V=0
\right\}.
\end{equation}
For a matrix $Z=(Z_{ij})_{1\leq i,j\leq3}$ written in $2\times2$ blocks, define
\[
\alpha(Z)
=
\frac14\bigl(\operatorname{tr}Z_{22}+\operatorname{tr}Z_{33}\bigr).
\]
The projection onto $\mathcal W$ along $\mathcal A$ is
\begin{equation}\label{eq:projection}
\pi_{\mathcal W}(Z)
=
\begin{pmatrix}
0&0&0\\
Z_{21}&Z_{22}-\alpha(Z)I_2&0\\
Z_{31}&Z_{32}&Z_{33}-\alpha(Z)I_2
\end{pmatrix}.
\end{equation}
In particular,
\[
M_6(\mathbb C)=\mathcal A\oplus\mathcal W.
\]

Let $\mathcal U_{\mathcal W}$ be the open subset of the Grassmannian consisting
of subspaces complementary to $\mathcal W$.  The graph map
\[
\Theta\colon
\operatorname{Hom}(\mathcal A,\mathcal W)
\longrightarrow
\mathcal U_{\mathcal W},
\quad
\psi\longmapsto
\Gamma_\psi
=
\{S+\psi(S):S\in\mathcal A\},
\]
is an affine-chart isomorphism and satisfies $\Theta(0)=\mathcal A$.  We therefore
identify
\begin{equation}\label{eq:tangent-chart}
T_{\mathcal A}\operatorname{Gr}(17,M_6(\mathbb C))
\cong
\operatorname{Hom}(\mathcal A,\mathcal W).
\end{equation}
For $\phi\in\operatorname{Hom}(\mathcal A,\mathcal W)$, write
\begin{equation}\label{eq:phi-blocks}
\phi(S)
=
\begin{pmatrix}
0&0&0\\
p(S)&u(S)&0\\
q(S)&r(S)&v(S)
\end{pmatrix},
\quad S\in\mathcal A,
\end{equation}
where $p,q,r,u,v\colon\mathcal A\to M_2(\mathbb C)$ are linear and
\begin{equation}\label{eq:trace-condition}
\operatorname{tr}u(S)+\operatorname{tr}v(S)=0
\quad(S\in\mathcal A).
\end{equation}

\subsection{Tangent space to the conjugacy orbit}

We next compute the tangent vectors arising from the differential of the
conjugation action in the affine coordinates above.
For $L,M,N\in M_2(\mathbb C)$, set
\[
X_{L,M,N}
=
\begin{pmatrix}
0&0&0\\
L&0&0\\
M&N&0
\end{pmatrix}
\]
and define
\[
\delta_{L,M,N}\colon\mathcal A\longrightarrow\mathcal W,
\quad
\delta_{L,M,N}(S)
=
\pi_{\mathcal W}([X_{L,M,N},S]).
\]
For $S=S(A,B,C,D,\lambda)$, set
\[
\beta_{L,M}(S)=\frac14\operatorname{tr}(LB+MC).
\]
Formulas \eqref{eq:projection} and \eqref{eq:A-definition} give
\begin{equation}\label{eq:delta-formula}
\begin{aligned}
\bigl(\delta_{L,M,N}(S)\bigr)_{21}
&=L(A-\lambda I_2)-DM,\\
\bigl(\delta_{L,M,N}(S)\bigr)_{31}
&=M(A-\lambda I_2),\\
\bigl(\delta_{L,M,N}(S)\bigr)_{32}
&=MB,\\
\bigl(\delta_{L,M,N}(S)\bigr)_{22}
&=LB-DN-\beta_{L,M}(S)I_2,\\
\bigl(\delta_{L,M,N}(S)\bigr)_{33}
&=MC+ND-\beta_{L,M}(S)I_2.
\end{aligned}
\end{equation}
All remaining $2\times2$ blocks of $\delta_{L,M,N}(S)$ are zero.

\begin{proposition}\label{prop:orbit-tangent}
Under the identification \eqref{eq:tangent-chart},
\[
T_{\mathcal A}\mathcal O
=
\{\delta_{L,M,N}:L,M,N\in M_2(\mathbb C)\}.
\]
The map
\[
M_2(\mathbb C)^3
\longrightarrow
T_{\mathcal A}\mathcal O,
\quad
(L,M,N)\longmapsto\delta_{L,M,N},
\]
is an isomorphism of vector spaces.
\end{proposition}

\begin{proof}
For $X\in M_6(\mathbb C)$, write
\[
X=Y+X_{L,M,N},
\quad
Y=
\begin{pmatrix}
X_{11} & X_{12} & X_{13}\\
0 & X_{22} & X_{23}\\
0 & 0 & X_{33}
\end{pmatrix},
\quad
X_{L,M,N}
=
\begin{pmatrix}
0 & 0 & 0\\
L & 0 & 0\\
M & N & 0
\end{pmatrix}.
\]
Then, under the identification \eqref{eq:tangent-chart}, the differential at
the identity of the orbit map
\[
\rho:\operatorname{GL}_6(\mathbb C)\longrightarrow\operatorname{Gr}(17,M_6(\mathbb C)),
\quad
g\longmapsto g\mathcal A g^{-1}
\]
is
\[
(d\rho)_{I_6}:M_6(\mathbb C)\to\operatorname{Hom}(\mathcal A,\mathcal W),\quad X\mapsto (S\mapsto \pi_{\mathcal W}([X,S])=\delta_{L,M,N}(S)).
\]
Hence
\[
\operatorname{Im}(d\rho)_{I_6}=\left\{\delta_{L,M,N}: L, M, N \in M_2(\mathbb C)\right\}\subseteq T_{\mathcal A}\mathcal O.
\]
By \eqref{eq:delta-formula}, the map
\[
M_2(\mathbb C)^3 \to\operatorname{Hom}(\mathcal{A},\mathcal W),\quad (L, M, N) \mapsto \delta_{L,M,N}
\]
is injective. It remains to prove the surjectivity. Since
\[
\operatorname{dim} \operatorname{Im}(d\rho)_{I_6}=\operatorname{dim} M_2(\mathbb C)^3=12 =  \operatorname{dim} \mathcal{O}=\operatorname{dim} T_{\mathcal A}\mathcal O,
\]
we have
\[
\operatorname{Im}(d\rho)_{I_6} = T_{\mathcal A}\mathcal O,
\]
as desired.
\end{proof}

\subsection{Linearization of the rank condition}

To determine which elements of $\operatorname{Hom}(\mathcal A,\mathcal W)$
are tangent to $\mathfrak X$, we linearize the determinantal rank condition.
The first lemma gives the tangent space to the rank-at-most-two determinantal
variety at a rank-two matrix.  The second is an auxiliary matrix identity that
will solve the final family of linearized constraints.

\begin{lemma}\label{lem:determinantal-tangent}
Let
\[
\mathcal D_2
=
\{C\in M_6(\mathbb C):\operatorname{rank}C\leq2\}.
\]
If $C_0\in\mathcal D_2$ has rank $2$, then
\begin{equation}\label{eq:determinantal-tangent}
T_{C_0}\mathcal D_2
=
\{C_1\in M_6(\mathbb C):
C_1(\ker C_0)\subseteq\operatorname{im}C_0\}.
\end{equation}
\end{lemma}

\begin{proof}
This is the standard tangent-space formula for determinantal varieties;
see \cite[Example~14.16]{Harris}.
\end{proof}

\begin{lemma}\label{lem:functional-identity}
Let $U,V\colon M_2(\mathbb C)\to M_2(\mathbb C)$ be linear maps satisfying
\begin{equation}\label{eq:functional-identity}
U(Y)Z-ZV(Y)+YV(Z)-U(Z)Y=0
\quad(Y,Z\in M_2(\mathbb C)).
\end{equation}
Then there exist $N\in M_2(\mathbb C)$ and a linear functional
$\tau\colon M_2(\mathbb C)\to\mathbb C$ such that
\begin{equation}\label{eq:functional-solution}
U(Y)=-YN+\tau(Y)I_2,
\quad
V(Y)=NY+\tau(Y)I_2.
\end{equation}
If, in addition,
\[
\operatorname{tr}U(Y)+\operatorname{tr}V(Y)=0
\quad(Y\in M_2(\mathbb C)),
\]
then $\tau=0$.
\end{lemma}

\begin{proof}
Set
\[
A_0=U(I_2),
\quad
B_0=V(I_2),
\quad
N=\frac{B_0-A_0}{2},
\quad
C_0=\frac{A_0+B_0}{2},
\]
and define
\[
\Phi(Y)=U(Y)+YN-C_0Y.
\]
Taking $Z=I_2$ in \eqref{eq:functional-identity} gives
\[
U(Y)-V(Y)+YB_0-A_0Y=0,
\]
and hence
\begin{equation}\label{eq:V-from-U}
V(Y)=U(Y)+YB_0-A_0Y.
\end{equation}
Substituting \eqref{eq:V-from-U} into \eqref{eq:functional-identity} gives
\begin{equation}\label{eq:Phi-identity}
[\Phi(Y),Z]+[Y,\Phi(Z)]
=
-C_0[Y,Z]-[Y,Z]C_0.
\end{equation}
Taking traces in \eqref{eq:Phi-identity} yields
\[
\operatorname{tr}(C_0[Y,Z])=0
\quad(Y,Z\in M_2(\mathbb C)).
\]
Since the commutators span $\mathfrak{sl}_2(\mathbb C)$ and the trace-orthogonal
complement of $\mathfrak{sl}_2(\mathbb C)$ is $\mathbb C I_2$, one has
$C_0=c_0I_2$ for some $c_0\in\mathbb C$.

Define $\Psi(Y)=\Phi(Y)+c_0Y=U(Y)+YN$.  Equation \eqref{eq:Phi-identity} becomes
\begin{equation}\label{eq:Psi-identity}
[\Psi(Y),Z]+[Y,\Psi(Z)]=0
\quad(Y,Z\in M_2(\mathbb C)).
\end{equation}
Taking $Z=I_2$ shows that $\Psi(I_2)$ is scalar.  Modulo scalar matrices, write
\[
\overline{\Psi(H)}=aH+bE+cF,\quad 
\overline{\Psi(E)}=dH+eE+fF,\quad 
\overline{\Psi(F)}=gH+hE+iF,
\]
where
\[
H=
\begin{pmatrix}1&0\\0&-1\end{pmatrix},
\quad
E=
\begin{pmatrix}0&1\\0&0\end{pmatrix},
\quad
F=
\begin{pmatrix}0&0\\1&0\end{pmatrix}.
\]
Applying \eqref{eq:Psi-identity} to $(H,E)$, $(H,F)$, and $(E,F)$ gives,
respectively,
\[
c=f=0,\quad e=-a,\quad b=h=0,\quad i=-a,\quad d=g=0,\quad e+i=0.
\]
Thus $a=0$, and every $\Psi(Y)$ is scalar.  Hence
$\Psi(Y)=\tau(Y)I_2$ for a linear functional $\tau$, which gives
\eqref{eq:functional-solution}.  Finally,
\[
\operatorname{tr}U(Y)+\operatorname{tr}V(Y)=4\tau(Y),
\]
so the additional trace condition implies $\tau=0$.
\end{proof}

\subsection{Proof of the tangent-space equality}

We now assemble the preceding ingredients.  Starting with an arbitrary tangent
vector $\phi\in T_{\mathcal A}\mathfrak X$, we apply the determinantal tangent
condition to a small collection of rank-two test commutators.  The first two
steps determine the components that agree with a tangent vector to $\mathcal O$
and subtract that tangent vector; the final step reduces the remaining terms to
Lemma~\ref{lem:functional-identity}.

For $X\in M_2(\mathbb C)$, define
\[
\begin{aligned}
\mathsf A(X)&=S(X,0,0,0,0),\\
\mathsf B(X)&=S(0,X,0,0,0),\\
\mathsf C(X)&=S(0,0,X,0,0),\\
\mathsf D(X)&=S(0,0,0,X,0),\\
\Lambda&=S(0,0,0,0,1).
\end{aligned}
\]
Then
\begin{equation}\label{eq:A-decomposition}
\mathcal A
=
\mathsf A(M_2(\mathbb C))
\oplus\mathsf B(M_2(\mathbb C))
\oplus\mathsf C(M_2(\mathbb C))
\oplus\mathsf D(M_2(\mathbb C))
\oplus\mathbb C\Lambda.
\end{equation}
For any component map $f\in\{p,q,r,u,v\}$, write
$f_{\mathsf A}=f\circ\mathsf A$, and similarly for
$\mathsf B,\mathsf C,\mathsf D$.

\begin{theorem}\label{thm:tangent-space-equality}
The tangent space of $\mathfrak X$ at $\mathcal A$ is the tangent space to
$\mathcal O$:
\[
T_{\mathcal A}\mathfrak X
=
T_{\mathcal A}\mathcal O.
\]
Consequently,
\[
\dim_{\mathbb C}T_{\mathcal A}\mathfrak X=12.
\]
\end{theorem}

\begin{proof}
The inclusion $T_{\mathcal A}\mathcal O\subseteq T_{\mathcal A}\mathfrak X$
follows from $\mathcal O\subseteq\mathfrak X$.  For the reverse inclusion, let
$\phi\in\operatorname{Hom}(\mathcal A,\mathcal W)$ represent a tangent vector to
$\mathfrak X$ through \eqref{eq:tangent-chart}. For
\[
S=S(A,B,C,D,\lambda),
\quad
T=S(A',B',C',D',\mu),
\]
set
\[
C_0=[S,T],
\quad
C_1=[\phi(S),T]+[S,\phi(T)].
\]
For fixed $S$ and $T$, define the regular map
\[
\mathcal C_{S,T}\colon
\operatorname{Hom}(\mathcal A,\mathcal W)
\longrightarrow M_6(\mathbb C),
\quad
\psi\longmapsto[S+\psi(S),T+\psi(T)].
\]
It satisfies
\[
\mathcal C_{S,T}\bigl(\Theta^{-1}(\mathfrak X\cap\mathcal U_{\mathcal W})\bigr)
\subseteq\mathcal D_2.
\]
Therefore, whenever $\operatorname{rank}C_0=2$, its differential at $0$ and
Lemma~\ref{lem:determinantal-tangent} give
\begin{equation}\label{eq:kernel-image-condition}
C_1(\ker C_0)\subseteq\operatorname{im}C_0.
\end{equation}
Direct block multiplication gives the lower two block rows of $C_1$:
\begin{equation}\label{eq:linearized-blocks}
\begin{aligned}
(C_1)_{21}
&=p(S)(A'-\mu I_2)-D'q(S)
  -p(T)(A-\lambda I_2)+Dq(T),\\
(C_1)_{22}
&=p(S)B'-D'r(S)+Dr(T)-p(T)B,\\
(C_1)_{23}
&=p(S)C'+u(S)D'-D'v(S)
  +Dv(T)-p(T)C-u(T)D,\\
(C_1)_{31}
&=q(S)(A'-\mu I_2)-q(T)(A-\lambda I_2),\\
(C_1)_{32}
&=q(S)B'-q(T)B,\\
(C_1)_{33}
&=q(S)C'+r(S)D'-q(T)C-r(T)D.
\end{aligned}
\end{equation}
\medskip
\noindent\emph{Step 1: test commutators and the resulting restrictions.}
For brevity, write $\mathcal E_{ij}(X)$ for the $3\times3$ block matrix
whose $(i,j)$-block is $X$ and whose remaining blocks are zero.  Set
\begin{equation}\label{eq:L-M-definition}
L=-p(\Lambda),
\qquad
M=-q(\Lambda).
\end{equation}
In each row of Table~\ref{tab:step1-tests}, all displayed variables $X$ and $Y$ are initially assumed
invertible, so that $\operatorname{rank}C_0=2$ in every row.  Hence
\eqref{eq:kernel-image-condition} determines which entries of $C_1$ must
vanish, and substitution into \eqref{eq:linearized-blocks} gives the
restrictions in the final column.  The rows are read in order, and restrictions
obtained from preceding rows are used when simplifying subsequent rows.

Each variable-dependent restriction is linear in the relevant variable.
Since $\operatorname{GL}_2(\mathbb C)$ spans $M_2(\mathbb C)$, these
restrictions extend to arbitrary $X,Y\in M_2(\mathbb C)$.

\begin{center}
\begin{samepage}
\begingroup
\setlength{\abovecaptionskip}{0pt}
\setlength{\belowcaptionskip}{2pt}
\captionof{table}{Test commutators and resulting restrictions in Step~1.}
\label{tab:step1-tests}
\footnotesize
\setlength{\tabcolsep}{5pt}
\renewcommand{\arraystretch}{1.18}
\setlength{\extrarowheight}{1pt}
\setlength{\aboverulesep}{0.45ex}
\setlength{\belowrulesep}{0.55ex}
\begin{tabular}{@{}
p{0.37\textwidth}
p{0.14\textwidth}
p{0.42\textwidth}
@{}}
\toprule
\textbf{Test pair $(S,T)$}
&
\textbf{$C_0=[S,T]$}
&
\textbf{Resulting restrictions}
\\
\midrule

$\bigl(\mathsf B(X),\Lambda\bigr)$
&
$\mathcal E_{12}(X)$
&
$\begin{aligned}
p_{\mathsf B}(X)&=0,\\
q_{\mathsf B}(X)&=0
\end{aligned}$
\\

\midrule

$\bigl(\mathsf C(X),\Lambda\bigr)$
&
$\mathcal E_{13}(X)$
&
$\begin{aligned}
p_{\mathsf C}(X)&=0,\\
q_{\mathsf C}(X)&=0
\end{aligned}$
\\

\midrule

$\bigl(\mathsf A(X)+\mathsf B(I_2),\Lambda\bigr)$
&
$\mathcal E_{12}(I_2)$
&
$\begin{aligned}
p_{\mathsf A}(X)&=LX,\\
q_{\mathsf A}(X)&=MX
\end{aligned}$
\\

\midrule

$\bigl(\mathsf B(I_2)+\mathsf D(Y),\Lambda\bigr)$
&
$\mathcal E_{12}(I_2)$
&
$\begin{aligned}
p_{\mathsf D}(Y)&=-YM,\\
q_{\mathsf D}(Y)&=0,\\
Yv(\Lambda)&=u(\Lambda)Y,\\
r(\Lambda)&=0
\end{aligned}$
\\

\midrule

$\bigl(\mathsf B(X),\mathsf D(I_2)\bigr)$
&
$\mathcal E_{13}(X)$
&
$r_{\mathsf B}(X)=MX$
\\

\midrule

$\bigl(\mathsf A(X)+\mathsf B(I_2),\mathsf D(I_2)\bigr)$
&
$\mathcal E_{13}(I_2)$
&
$r_{\mathsf A}(X)=0$
\\

\midrule

$\bigl(\mathsf B(I_2)+\mathsf C(X),\mathsf D(I_2)\bigr)$
&
$\mathcal E_{13}(I_2)$
&
$r_{\mathsf C}(X)=0$
\\

\bottomrule
\end{tabular}
\endgroup
\end{samepage}
\end{center}
The fourth row of Table~\ref{tab:step1-tests} also gives
\[
Yv(\Lambda)=u(\Lambda)Y
\qquad
(Y\in M_2(\mathbb C)).
\]
Taking $Y=I_2$ first gives
$u(\Lambda)=v(\Lambda)$.  The same identity then shows that this common
matrix commutes with every element of $M_2(\mathbb C)$, and hence is scalar.
The trace condition \eqref{eq:trace-condition} therefore gives
\begin{equation}\label{eq:lambda-vanishing}
u(\Lambda)=v(\Lambda)=0.
\end{equation}
Combining the restrictions in Table~\ref{tab:step1-tests} with
\eqref{eq:A-decomposition}, we obtain
\begin{equation}\label{eq:pq-and-r}
\begin{aligned}
p(S(A,B,C,D,\lambda))
&=L(A-\lambda I_2)-DM,\\
q(S(A,B,C,D,\lambda))
&=M(A-\lambda I_2),\\
r_{\mathsf A}(X)
&=0,\\
r_{\mathsf B}(X)
&=MX,\\
r_{\mathsf C}(X)
&=0,\\
r(\Lambda)
&=0.
\end{aligned}
\end{equation}
Since
\[
\delta_{L,M,0}\in T_{\mathcal A}\mathcal O
\subseteq T_{\mathcal A}\mathfrak X,
\]
we may replace $\phi$ by $\phi-\delta_{L,M,0}$ and therefore assume
\begin{equation}\label{eq:phi-after-subtraction}
p=q=0,
\quad
r_{\mathsf A}=r_{\mathsf B}=r_{\mathsf C}=0,
\quad
r(\Lambda)=0,
\quad
u(\Lambda)=v(\Lambda)=0.
\end{equation}

\medskip
\noindent\emph{Step 2: the diagonal blocks on the
$\mathsf A$-, $\mathsf B$-, and $\mathsf C$-summands.}
We again use \eqref{eq:kernel-image-condition} together with
\eqref{eq:phi-after-subtraction}.  In Table~\ref{tab:step2-tests}, we initially take
$X,Y\in\operatorname{GL}_2(\mathbb C)$.  Then in each row, $C_0$ has rank $2$
and image $E_1$.  The kernel-image condition forces the relevant lower block of
$C_1$ to vanish, and substitution into \eqref{eq:linearized-blocks} gives the
restriction in the final column.  The rows are read in order, and restrictions
obtained from preceding rows are used when simplifying subsequent rows.  Since
each resulting identity is bilinear in $X$ and $Y$, and
$\operatorname{GL}_2(\mathbb C)$ spans $M_2(\mathbb C)$, these identities
extend to all $X,Y\in M_2(\mathbb C)$.
\begin{center}
\captionof{table}{Diagonal-block restrictions in Step~2.}
\label{tab:step2-tests}
\begingroup
\footnotesize
\setlength{\tabcolsep}{5pt}
\renewcommand{\arraystretch}{1.25}
\setlength{\extrarowheight}{2pt}
\setlength{\aboverulesep}{0.55ex}
\setlength{\belowrulesep}{0.75ex}
\begin{tabular}{@{}
p{0.39\textwidth}
p{0.25\textwidth}
p{0.29\textwidth}
@{}}
\toprule
\textbf{Test pair $(S,T)$}
&
\textbf{$C_0=[S,T]$}
&
\textbf{Resulting restrictions}
\\
\midrule

$\bigl(\mathsf B(X),\Lambda+\mathsf D(Y)\bigr)$
&
$\mathcal E_{12}(X)+\mathcal E_{13}(XY)$
&
$u_{\mathsf B}(X)Y=Yv_{\mathsf B}(X)$
\\

\midrule

$\bigl(\mathsf A(X)+\mathsf B(I_2),
\Lambda+\mathsf D(Y)\bigr)$
&
$\mathcal E_{12}(I_2)+\mathcal E_{13}(Y)$
&
$u_{\mathsf A}(X)Y=Yv_{\mathsf A}(X)$
\\

\midrule

$\bigl(\mathsf B(I_2)+\mathsf C(X),
\Lambda+\mathsf D(Y)\bigr)$
&
$\mathcal E_{12}(I_2)+\mathcal E_{13}(X+Y)$
&
$u_{\mathsf C}(X)Y=Yv_{\mathsf C}(X)$
\\

\bottomrule
\end{tabular}
\endgroup
\end{center}
For the first row, taking $Y=I_2$ shows that
$u_{\mathsf B}(X)=v_{\mathsf B}(X)$.  The displayed identity then implies
that this common matrix is scalar, and the trace condition
\eqref{eq:trace-condition} forces it to vanish.  Since the invertible matrices
span $M_2(\mathbb C)$,
\begin{equation}\label{eq:B-diagonal-zero}
u_{\mathsf B}=v_{\mathsf B}=0.
\end{equation}
This restriction is used when simplifying the second and third rows of
Table~\ref{tab:step2-tests}.  Applying the same argument to those two rows gives
\begin{equation}\label{eq:ABC-diagonal-zero}
u_{\mathsf A}=v_{\mathsf A}
=
u_{\mathsf C}=v_{\mathsf C}
=0.
\end{equation}

\medskip
\noindent\emph{Step 3: the $\mathsf D$-summand.}
By \eqref{eq:phi-after-subtraction}, \eqref{eq:B-diagonal-zero}, and
\eqref{eq:ABC-diagonal-zero}, the only possibly nonzero components of $\phi$ are
$r_{\mathsf D}$, $u_{\mathsf D}$, and $v_{\mathsf D}$.  Define
\[
R(Y)=r_{\mathsf D}(Y),
\qquad
U(Y)=u_{\mathsf D}(Y),
\qquad
V(Y)=v_{\mathsf D}(Y).
\]
Let
\[
E=
\begin{pmatrix}
0&1\\
0&0
\end{pmatrix},
\qquad
F=
\begin{pmatrix}
0&0\\
1&0
\end{pmatrix},
\qquad
H=[E,F]
=
\begin{pmatrix}
1&0\\
0&-1
\end{pmatrix}.
\]
For arbitrary $Y,Z\in M_2(\mathbb C)$, the test pair in
Table~\ref{tab:step3-tests} has $C_0$ of rank $2$, with image $E_1$ and
kernel $E_2\oplus E_3$.
All restrictions established in Steps~1 and~2 are in force.  Hence
\eqref{eq:kernel-image-condition} forces the lower blocks of $C_1$ in the
second and third block columns to vanish.  Substitution into
\eqref{eq:linearized-blocks} gives the identities in the final column.

\begin{center}
\captionof{table}{The remaining test commutator and identities in Step~3.}
\label{tab:step3-tests}
\begingroup
\footnotesize
\setlength{\tabcolsep}{5pt}
\renewcommand{\arraystretch}{1.25}
\setlength{\extrarowheight}{2pt}
\setlength{\aboverulesep}{0.55ex}
\setlength{\belowrulesep}{0.75ex}
\begin{tabular}{@{}
p{0.30\textwidth}
p{0.14\textwidth}
p{0.49\textwidth}
@{}}
\toprule
\textbf{Test pair $(S,T)$}
&
\textbf{$C_0=[S,T]$}
&
\textbf{Resulting identities}
\\
\midrule

$\bigl(
\mathsf A(E)+\mathsf D(Y),
\mathsf A(F)+\mathsf D(Z)
\bigr)$
&
$\mathcal E_{11}(H)$
&
$\begin{gathered}
ZR(Y)=YR(Z),\\
R(Y)Z=R(Z)Y,\\
U(Y)Z-ZV(Y)+YV(Z)-U(Z)Y=0
\end{gathered}$
\\

\bottomrule
\end{tabular}
\endgroup
\end{center}
Taking $Z=I_2$ in the first two identities gives
\[
R(Y)=YR(I_2)=R(I_2)Y.
\]
Hence $R(I_2)$ is scalar, say $R(I_2)=cI_2$, and therefore
$R(Y)=cY$.  Substitution into either identity yields
$c[Z,Y]=0$ for all $Y,Z\in M_2(\mathbb C)$, so $c=0$.  Thus
\[
r_{\mathsf D}=0.
\]
The third identity is precisely the hypothesis of
Lemma~\ref{lem:functional-identity} for
$U=u_{\mathsf D}$ and $V=v_{\mathsf D}$.  Moreover,
\eqref{eq:trace-condition}, applied to $\mathsf D(Y)$, gives
\[
\operatorname{tr}U(Y)+\operatorname{tr}V(Y)=0.
\]
The lemma therefore yields some $N\in M_2(\mathbb C)$ such that
\[
U(Y)=-YN,
\qquad
V(Y)=NY
\qquad
(Y\in M_2(\mathbb C)).
\]
It follows from \eqref{eq:delta-formula} that
\[
\phi=\delta_{0,0,N},
\]
and hence Proposition~\ref{prop:orbit-tangent} gives
$\phi\in T_{\mathcal A}\mathcal O$. This proves
\[
T_{\mathcal A}\mathfrak X=T_{\mathcal A}\mathcal O.
\]
The dimension statement follows from
Proposition~\ref{prop:orbit-tangent}.
\end{proof}

\section{Global classification}

Theorem~\ref{thm:tangent-space-equality} gives the equality
\[
T_{\mathcal A}\mathfrak X=T_{\mathcal A}\mathcal O.
\]
We first use this equality to show that $\mathcal O$ and
$\mathcal O^{\mathsf T}$ are irreducible components of $\mathfrak X$.  To
exclude any further component, we use the Borel fixed-point theorem to produce
an upper-triangular fixed point on an arbitrary component and then apply the
classification of Omladi\v{c}, Radjavi, and \v{S}ivic.

\begin{proposition}\label{prop:orbit-components}
The subvarieties $\mathcal O$ and $\mathcal O^{\mathsf T}$ are irreducible
components of $\mathfrak X$.  Every point of either orbit is nonsingular in
$\mathfrak X$, and no other irreducible component meets either orbit.
\end{proposition}

\begin{proof}
Since $\mathcal O\subseteq\mathfrak X$ and $\mathcal A\in\mathcal O$,
Theorem~\ref{thm:tangent-space-equality} gives
\[
12=\dim\mathcal O
\leq\dim_{\mathcal A}\mathfrak X
\leq\dim_{\mathbb C}T_{\mathcal A}\mathfrak X
=12.
\]
Thus $\mathcal A$ is nonsingular and lies on a unique irreducible component
$C$ of $\mathfrak X$.  Since $\mathcal O$ is closed and irreducible,
$\mathcal O\subseteq C$.  Moreover,
\[
\dim C=\dim_{\mathcal A}\mathfrak X=12=\dim\mathcal O,
\]
and hence $C=\mathcal O$.

Conjugation by any element of $\operatorname{GL}_6(\mathbb C)$ is an algebraic
automorphism of $\mathfrak X$, so the same conclusions hold at every point of
$\mathcal O$.  Transposition is also an algebraic automorphism of $\mathfrak X$,
because
\[
[S^{\mathsf T},T^{\mathsf T}]=-[S,T]^{\mathsf T}.
\]
It sends $\mathcal O$ to $\mathcal O^{\mathsf T}$, proving the corresponding
statements for the transpose orbit.
\end{proof}

The only external classification input needed for the global argument is the
following specialization of \cite[Proposition~11]{ORS}.

\begin{theorem}[Omladi\v{c}--Radjavi--\v{S}ivic]\label{thm:ORS-special}
Let $\mathcal V\subseteq M_6(\mathbb C)$ be a complex linear subspace of
dimension $17$ satisfying:
\begin{enumerate}
\item $\operatorname{rank}[S,T]\leq2$ for all $S,T\in\mathcal V$;
\item $P\mathcal VP^{-1}=\mathcal V$ for every invertible upper-triangular matrix
$P$.
\end{enumerate}
Then $\mathcal V$ or $\mathcal V^{\mathsf T}$ is conjugate to $\mathcal A$.
\end{theorem}

We apply this fixed-point classification to an arbitrary irreducible component
of $\mathfrak X$.  The following componentwise reduction is the refinement of
\cite[Lemma~8]{ORS} needed for the classification of the whole locus.

\begin{proposition}\label{prop:fixed-point-reduction}
Every irreducible component of $\mathfrak X$ meets
$\mathcal O\cup\mathcal O^{\mathsf T}$.
\end{proposition}

\begin{proof}
Let $C$ be an irreducible component of $\mathfrak X$.  Since $\operatorname{GL}_6(\mathbb C)$ is irreducible,
$\operatorname{GL}_6(\mathbb C)\times C$ is irreducible.  The closure of its
image under the regular action map is
an irreducible closed subset of $\mathfrak X$ containing $C$.  By maximality of
$C$, this closure is $C$.  Hence $gC=C$ for every $g\in\operatorname{GL}_6(\mathbb C)$.

Let $B\subseteq\operatorname{GL}_6(\mathbb C)$ be the subgroup of invertible upper-triangular matrices.  The
component $C$ is projective by Proposition~\ref{prop:closed}, and it is stable under
the connected solvable group $B$.  The Borel fixed-point theorem
\cite[Theorem~10.4]{Borel} therefore gives a $B$-fixed point
$\mathcal V\in C$.  This point satisfies both hypotheses of
Theorem~\ref{thm:ORS-special}; hence
\[
\mathcal V\in\mathcal O\cup\mathcal O^{\mathsf T}.
\]
\end{proof}

The preceding two propositions now determine all irreducible components of
$\mathfrak X$.

\begin{theorem}\label{thm:locus}
\[
\mathfrak X=\mathcal O\cup\mathcal O^{\mathsf T}.
\]
\end{theorem}

\begin{proof}
The inclusion from right to left follows from Proposition~\ref{prop:orbit} and
invariance under conjugation and transposition.  Conversely, let
$x\in\mathfrak X$ and let $C$ be an irreducible component containing $x$.
By Proposition~\ref{prop:fixed-point-reduction}, the component $C$ meets one of
the two orbits $\mathcal O$ and $\mathcal O^{\mathsf T}$.  Proposition~\ref{prop:orbit-components} implies that the
orbit it meets is the unique irreducible component through the intersection point.
Thus $C$ equals that orbit, and
$x\in\mathcal O\cup\mathcal O^{\mathsf T}$.
\end{proof}

\begin{proof}[Proof of Theorem~\ref{thm:main}]
By Theorem~\ref{thm:locus}, either $\mathcal V$ is conjugate to $\mathcal A$ or
$\mathcal V$ is conjugate to $\mathcal A^{\mathsf T}$.  In the second case, write
$\mathcal V=Q\mathcal A^{\mathsf T}Q^{-1}$.  Then
\[
\mathcal V^{\mathsf T}
=
Q^{-\mathsf T}\mathcal A Q^{\mathsf T},
\]
so $\mathcal V^{\mathsf T}$ is conjugate to $\mathcal A$.
\end{proof}

\section{Distinctness of the two components}

Theorem~\ref{thm:locus} gives
\[
\mathfrak X=\mathcal O\cup\mathcal O^{\mathsf T},
\]
but the two orbits could still coincide a priori.  It remains to show that they
are distinct.  Since conjugate matrix algebras are isomorphic, it suffices to
distinguish $\mathcal A$ from $\mathcal A^{\mathsf T}$ by an algebra-isomorphism
invariant.  We use the bimodule structure of the radical square.

\begin{proposition}\label{prop:distinct}
The algebras $\mathcal A$ and $\mathcal A^{\mathsf T}$ are not isomorphic.
Consequently, $\mathcal O$ and $\mathcal O^{\mathsf T}$ are distinct.
\end{proposition}

\begin{proof}
Let $\mathcal B$ be either $\mathcal A$ or $\mathcal A^{\mathsf T}$, and set
$J_{\mathcal B}=\operatorname{rad}(\mathcal B)$.  Since
$J_{\mathcal B}^3=0$, multiplication makes $J_{\mathcal B}^2$ an
$(\mathcal B/J_{\mathcal B})$-bimodule.  The semisimple quotient is
\[
\mathcal B/J_{\mathcal B}\cong M_2(\mathbb C)\oplus\mathbb C.
\]
Let $\bar e_{\mathcal B}$ be the identity of the unique noncommutative simple
summand.  The ordered pair
\[
\iota(\mathcal B)
=
\left(
\dim_{\mathbb C}(\bar e_{\mathcal B}J_{\mathcal B}^2),
\dim_{\mathbb C}(J_{\mathcal B}^2\bar e_{\mathcal B})
\right)
\]
is preserved by algebra isomorphisms.

For $\mathcal A$, the space $J_{\mathcal A}^2$ is the $(1,3)$ block.  With
$e=\operatorname{diag}(I_2,0,0)$, every $R\in J_{\mathcal A}^2$ satisfies
$eR=R$ and $Re=0$.  Hence
\[
\iota(\mathcal A)=(4,0).
\]
For $\mathcal A^{\mathsf T}$, the radical square is the $(3,1)$ block, so
\[
\iota(\mathcal A^{\mathsf T})=(0,4).
\]
Thus $\mathcal A$ and $\mathcal A^{\mathsf T}$ are not isomorphic.  Conjugate
matrix algebras are isomorphic, so their conjugacy orbits are distinct.
\end{proof}

\begin{corollary}\label{cor:components}
The algebraic set $\mathfrak X$ is the disjoint union of exactly two
nonsingular projective irreducible components, each isomorphic to
$\operatorname{Fl}(2,4;6)$.
\end{corollary}

\begin{proof}
By Theorem~\ref{thm:locus}, $\mathfrak X$ is the union of
$\mathcal O$ and $\mathcal O^{\mathsf T}$.  Proposition~\ref{prop:distinct} shows
that these are distinct orbits, hence disjoint.  Their remaining properties follow
from Propositions~\ref{prop:orbit} and \ref{prop:orbit-components}.
\end{proof}

\section*{Declaration of AI use}

This manuscript was developed with extensive use of OpenAI's ChatGPT. ChatGPT played a major role in the selection and refinement of the research problem, the iterative formulation of the main theorem, literature searches, exploration of proof strategies, development and assembly of the proof, verification of calculations, checking the applicability and hypotheses of cited results, verification of citations and bibliographic information, and the drafting of substantially the entire manuscript. The author's role consisted primarily of reviewing, correcting where necessary, and approving the final mathematical arguments and exposition. The author takes full responsibility for the correctness and contents of the final manuscript.

\end{document}